\documentclass[12pt,reqno,letterpaper]{amsart}
\usepackage{mathrsfs}

\usepackage[T1]{fontenc}
\usepackage{enumitem}
\usepackage{graphicx}
\graphicspath{ {./images/} }
\usepackage{esint}
\usepackage{hyperref}
\hypersetup{
  colorlinks   = true,
  citecolor    = blue,
  linkcolor    = blue
}

\usepackage{amsmath,amsfonts,amsthm,color,tikz, comment, xcolor, xfrac,amssymb}
\usepackage{mathtools}

\usepackage[normalem]{ulem}
\usepackage{mathrsfs}

\usepackage{pdfsync}
\usepackage[font={scriptsize}]{caption}
\usepackage{environ}

\usepackage[left=1in, right=1in, top=1in,bottom=1in]{geometry}
\numberwithin{equation}{section}

\def\XXint#1#2#3{{\setbox0=\hbox{$#1{#2#3}{\int}$ }
\vcenter{\hbox{$#2#3$ }}\kern-.6\wd0}}

\usepackage{fix-cm}

\newcommand{\1}{{\bf 1}}

\newcommand{\R}{\mathbb R}
\newcommand{\N}{\mathbb N}

\newcommand{\T}{\mathbb T}
\newcommand{\Z}{\mathbb Z}
\renewcommand{\P}{\mathbb{P}}

\newtheorem{theorem}{Theorem}[section]

\newtheorem{definition}[theorem]{Definition}

\newtheorem{lemma}[theorem]{Lemma}

\newtheorem{proposition}[theorem]{Proposition}

\theoremstyle{remark}
\newtheorem{remark}[theorem]{Remark}

\theoremstyle{remark}

\newcommand{\bean}{\begin{eqnarray*}}
\newcommand{\eean}{\end{eqnarray*}}
\newcommand{\ben}{\begin{enumerate}}
\newcommand{\een}{\end{enumerate}}
\newcommand{\beq}{\begin{equation}}
\newcommand{\eeq}{\end{equation}}

\allowdisplaybreaks

\begin{document}

\author{ Mandon Pathak}

\title{Nontrivial Weak Solutions of the Stationary KdV equation in Sharp $L^p$ Spaces}

\begin{abstract}
    In this paper we utilize a convex integration scheme to construct non-trivial solutions to the stationary KdV equation which lie in $L^p(\T)$, $p < 2$. In addition, we demonstrate this result is sharp in the sense that if $u \in L^2(\T)$ is a weak solution then $u \in C^\infty(\T)$.
\end{abstract}

\maketitle

\section{Introduction}

\subsection{Motivation and Background}

The stationary Korteweg--de Vries (KdV) equation
\begin{equation}\label{eq:stat_eqn}
6uu' - u^{(3)} = 0
\end{equation}
for $u:\T \to \R$ arises as the time-independent counterpart of the classical KdV equation, which models nonlinear dispersive phenomena such as shallow water waves and plasma oscillations. The KdV equation is a completely integrable system with a rich mathematical structure, including soliton solutions and an infinite hierarchy of conserved quantities, making it a central object in both applied and theoretical analysis of nonlinear dispersive partial differential equations.

Recent years have seen many exciting developments in the theory surrounding the KdV equation, mostly concerning questions of wellposedness. To give a very brief overview, in \cite{Molinet2011} and \cite{Molinet2012} Molinet shows that KdV is ill-posed for initial data lying in $H^{-s}(\T)$ and $H^{-s}(\R)$ for $s > 1$. Kappeler and Topalov \cite{KT} then showed that KdV is wellposed in $H^{-1}(\T)$, and finally in \cite{KV} Killip and Vi\c{s}an established the same threshold for KdV posed on the real line. We also mention the work of Christ \cite{Christ1} where he constructs a nonzero solution to KdV which has zero initial data and lies in $C_tH^{-\epsilon}_x([0,T] \times \T)$, $\epsilon > 0$. This does not conflict with the result of \cite{KT} since Christ's notion of solution is weaker than what is considered in \cite{KT}, \cite{KV}, \cite{Molinet2011}, and \cite{Molinet2012}. This demonstrates the important role that low regularity solutions play in the study of KdV. In contrast, low regularity solutions of stationary KdV have to this point remained largely unexplored.

In this work, we investigate weak solutions of the stationary KdV equation in $L^p(\mathbb{T})$ for $p<2$. Our main result establishes a sharp dichotomy: 
\begin{itemize}
    \item for $p<2$, there exist nontrivial weak solutions that are not smooth, and 
    \item for $p=2$, every weak solution is necessarily smooth (see \cite{DN}, \cite{FN2016}, and \cite{GS} for examples of the construction of such smooth solutions).
\end{itemize}
This result identifies the precise threshold in $L^p$ regularity separating smooth from potentially singular solutions and demonstrates that stationary KdV admits a rich set of low-regularity solutions.

One technicality is how exactly to define the notion of a weak solution when $u \in L^p(\T)$, $p < 2$. Since the nonlinearity in ~\eqref{eq:stat_eqn} is quadratic, then in order for the classical notion of a weak solution to be well defined, this would require that $u \in L^2(\T)$. But from Theorem \ref{thm:main_rig}, this would immediately imply that $u$ is smooth. At this low regularity, the notion of a weak solution for non-stationary KdV used in for instance \cite{KT} and \cite{KV} depends crucially on there being a dependence on the time parameter, and thus seems ill suited for the stationary setting. So in order to overcome this difficulty, we take inspiration from \cite{ABGN}\footnote{In \cite{LR}, Lemari\'{e}-Rieusset proposes an alternative definition for the non-linearity $u\otimes u$ in the case when $u \in L^p(\T^2)$, $p < 2$, which differs slightly from what is offered in \cite{ABGN}.} where the notion of a weak paraproduct solution\footnote{In \cite{CH} these are referred to as \textit{singular solutions}.} to the 2D stationary Navier-Stokes equation is introduced (see Definition \ref{def:weak_para_soln}). The two rigidity results in this paper and for 2D stationary Navier-Stokes should be compared. For 2D stationary Navier-Stokes, it is known that if $u \in L^p(\T^2)$, $p > 2$, is a weak solution with zero mean, then $u = 0$ \cite{ABGN, LR}. On the other hand, Theorem \ref{thm:main_rig} gives a sharp result which also includes the endpoint $p=2$. The downside is for stationary KdV we cannot conclude that smooth mean zero solutions must be identically $0$. This result should also be compared to \cite[Lemma 1]{FN2016} where it is shown that if $u \in C^3(\T)$ is a solution of stationary KdV then $u$ is smooth.

The construction of these solutions relies on a \emph{convex integration }framework, first introduced by De Lellis and Sz\'{e}kelyhidi in \cite{DLS09} and \cite{DLS13} in the context of fluid equations.
Beyond establishing the existence of such solutions for the stationary KdV equation, this approach \emph{opens a new avenue for applying convex integration techniques to KdV-type equations and more general nonlinear dispersive PDEs}, potentially yielding further examples of low-regularity phenomena in integrable and non-integrable contexts. Previously, the method of convex integration was not able to be applied to the KdV equation or more general dispersive PDEs due to the loss of derivatives one would incur in what we call the \textit{dispersion error}:
$$
E_D = -w_{q+1}''.
$$
See Section \ref{section3} for a more detailed discussion of how this term arises. For the purposes of this discussion, let us assume that $w_{q+1}$ is frequency localized to scale $\lambda_{q+1} \gg 1$. In standard convex integration schemes, one is required to estimate this error either in an $L^p$ norm or a $C^\alpha$ norm \cite{BV2020}. Let us again assume we are using an $L^p$ norm and thus we have
\begin{equation}\label{eq:dis_error_est_intro}
    \Vert E_D \Vert_{L^p} \lesssim \lambda_{q+1}^2 \Vert w_{q+1} \Vert_{L^p}.
\end{equation}
In order for the convex integration scheme to close, we require that $\Vert E_D \Vert_{L^p}$ can be made small. In order for this to occur, ~\eqref{eq:dis_error_est_intro} implies that $\Vert w_{q+1} \Vert_{L^p}$ must be made small, in particular, much smaller than $\lambda_{q+1}^2$. The standard tool to achieve this is to utilize an \textit{intermittent building block}. Intermittency has played a key role in many convex integration constructions; see for instance  \cite{ABGN}, \cite{BBV}, \cite{BCV}, \cite{BMNV}, \cite{BV19}, \cite{CH}, \cite{CL2}, \cite{CheskidovLuo}, \cite{DS17}, \cite{Gismondi}, \cite{GR}, \cite{Luo}, \cite{MS}, \cite{NV}, \cite{Peng} and references therein. Unfortunately, the maximum amount of intermittency possible in one dimension only gives the rough estimate
\begin{equation}\label{eq:intermittent_est}
    \Vert w_{q+1} \Vert_{L^p} \lesssim \lambda_{q+1}^{\frac{1}{2} - \frac{1}{p}}.
\end{equation}
Combining ~\eqref{eq:dis_error_est_intro} and ~\eqref{eq:intermittent_est} shows there is no value of $p \geq 1$ which makes the error small for $\lambda_{q+1} \gg 1$. This is the same reason the scheme employed by Luo in \cite{Luo} requires the dimension to be at least $4$.

In order to sidestep this issue, we measure the error in a homogeneous Sobolev norm with a negative regularity index as was done in \cite{ABGN}, \cite{CH}, \cite{Gismondi}, and \cite{GR}. By choosing $s > 0$ large enough we have the estimate
\begin{equation*}
    \Vert E_D \Vert_{\dot{H}^{-s}} \lesssim \Vert w_{q+1} \Vert_{\dot{H}^{-s+2}} \lesssim \Vert w_{q+1} \Vert_{L^1} \lesssim \lambda_{q+1}^{-1/2}
\end{equation*}
which clearly is small for $\lambda_{q+1} \gg 1$. Morally this is the procedure we use, except we utilize a very small amount of intermittency in order to achieve the best possible Besov regularity. See Section \ref{section5} for more details.

These findings contribute to the broader understanding of solution regularity in nonlinear dispersive equations and illustrate the effectiveness of convex integration as a tool for constructing weak solutions with controlled $L^p$ properties. They also suggest a new direction for studying the interplay between integrability, dispersion, and low-regularity behavior in classical PDEs.

\subsection{Main Result}

Our first task is to extend the definition of a weak solution to ~\eqref{eq:stat_eqn} to $u$ lying in some Sobolev space with strictly negative regularity. To do this, we need to make sense of $u^2$ for $u$ potentially an arbitrary distribution. In general, there is not much that can be said about this product when both functions lie in a Sobolev space with negative regularity, but following \cite[Definition 1.1]{ABGN} we may offer the following definition of the mean free product:

\begin{definition}[\textbf{Paraproducts in $\dot H^s(\T)$}]\label{def:paras}
    Let $f,g$ be distributions, so that $\mathbb{P}_{2^j}(f), \mathbb{P}_{2^{j'}}(g)$ are well-defined for $j,j'\geq 0$.  We say that $\mathbb{P}_{\not = 0}(fg)$ is well-defined as a paraproduct in $\dot{H}^s(\T)$ for some $s \in \R$ if
    $$
        \sum_{j,j' \geq 0} \left\Vert \mathbb{P}_{\not = 0}\left(\mathbb{P}_{2^j}(f) \mathbb{P}_{2^{j'}}(g)\right) \right\Vert_{\dot{H}^s} < \infty \, .
    $$
    Then we define
    $$
        \mathbb{P}_{\not = 0}(fg) = \sum_{j,j' \geq 0} \mathbb{P}_{\not = 0} \left(\mathbb{P}_{2^j}(f) \mathbb{P}_{2^{j'}}(g)\right) \, ,
    $$
    since the right-hand side is an absolutely summable series in $\dot H^s(\T)$.
\end{definition}

With this definition, adapting \cite[Definition 1.2]{ABGN}, we may now define a weak solution to \eqref{eq:stat_eqn} which is valid for $u$ belonging to a Sobolev space of arbitrary regularity.

\begin{definition}[\textbf{Weak paraproduct solutions to ~\eqref{eq:stat_eqn}}] \label{def:weak_para_soln}
    If $u \in \dot{H}^s$, $s < 0$, we say $u$ is a weak paraproduct solution to the stationary KdV equation if there is $s' \in \R$ such that $\mathbb{P}_{\not = 0}(u^2)$ is well defined as a paraproduct in $\dot{H}^{s'}$ in the sense of the previous definition and
    $$
        -\left\langle u, \phi^{(3)} \right\rangle_{\dot{H}^{s},\dot{H}^{-s}} + 3\left\langle \mathbb{P}_{\not= 0}(u^2), \phi' \right\rangle_{\dot{H}^{s'}, \dot{H}^{-s'}} = 0
    $$
    for all smooth $\phi$.
\end{definition}

With this, we may state our main results.

\begin{theorem}[\textbf{Rigidity Result}]\label{thm:main_rig}
    If $u \in L^2(\T)$ solves ~\eqref{eq:stat_eqn} then $u \in C^\infty(\T)$.
\end{theorem}
\begin{proof}
Suppose $u \in L^2(\T)$ satisfies
$$
-\int_{\T} u \psi^{(3)} + 3\int_{\T} u^2 \psi' = 0
$$
for all $\psi \in C^\infty(\T)$. In the sense of distributions we have
$$
\left\langle \left(u'' - 3u^2\right)', \psi \right\rangle = \left\langle u^{(3)} - 3(u^2)', \psi \right\rangle = 0.
$$
And so, in the sense of distributions we have $u'' - 3u^2 = C$ for some $C \in \R$. Since $u \in L^2(\T)$, this implies $u \in W^{2,1}(\T)$. From the Sobolev embedding $W^{2,1}(\T) \subset C^1(\T)$ we conclude $u \in C^1(\T)$, and thus $u \in C^3(\T)$. From here we may either appeal to bootstrapping or \cite[Lemma 1]{FN2016} to deduce $u \in C^\infty(\T)$.
\end{proof}

\begin{theorem}[\textbf{Flexibility Result}]\label{thm:main_flex}
    There is
    $$
    u \in \bigcap_{0 < \epsilon < 1} \left(\dot{B}_{\infty,\infty}^{-\epsilon}(\T) \cap L^{2-\epsilon}(\T)\right) \setminus L^2(\T)
    $$
    such that
    \begin{enumerate}
        \item [(a)] $\mathbb{P}_{\not=0}(u^2)$ exists as a paraproduct in $\dot{H}^{-s}(\T)$ for all $s$ large enough;
        \item [(b)] $u$ solves ~\eqref{eq:stat_eqn} in the sense of Definition \ref{def:weak_para_soln}.
    \end{enumerate}
\end{theorem}

\begin{remark}
    In Theorem \ref{thm:main_flex} it suffices to take $s > 5/2$, but this exact value is unimportant for our purposes.
\end{remark}

The rest of this paper is concerned with the proof of Theorem \ref{thm:main_flex}.

\subsection{Outline of Paper}
In Section \ref{Section2} we review some material related to Littlewood-Paley theory and the function spaces we will be considering in this paper. We also introduce the intermittent building block we utilize in the construction. In Section \ref{section3} we give a brief overview of the convex integration scheme to follow. In Section \ref{Section4} we state the main inductive proposition (Proposition \ref{prop:ind}) and use it to prove Theorem \ref{thm:main_flex}. We then provide the definition of our highly oscillatory perturbation and use it to prove Proposition \ref{prop:ind} in Section \ref{section5}.

\subsection{Acknowledgments} The author is thankful for the discussion with Nicholas Gismondi.

\section{Background Theory and Technical Lemmas}\label{Section2}

The following on Littlewood-Paley theory, Fourier multiplier operators, and function spaces can be found in \cite{BCD} and \cite{Grafakos}. The exact formulation of Definition \ref{def:projs} is based off \cite[Definition 2.1]{ABGN}.

\begin{definition}[\textbf{Littlewood-Paley projectors}]\label{def:projs}
    There exists $\varphi \colon \R\to[0,1]$, smooth, radially symmetric, and compactly supported in $\{\xi : 6/7 \leq |\xi|\leq 2\}$ such that $\varphi(\xi) = 1$ on $\{\xi : 1 \leq |\xi| \leq 12/7\}$,
    \begin{equation}
        \sum_{j\geq 0}\varphi(2^{-j}\xi)=1 \hspace{0.25cm} \text{ for all } \hspace{0.25cm} |\xi|\geq 1, \notag
    \end{equation}
    and $\operatorname{supp}\varphi_{j}\cap \operatorname{supp} \varphi_{j'}=\emptyset$ for all $|j-j'|\geq 2$, where $\varphi_j(\cdot) = \varphi(2^{-j}\cdot)$. We define the projection of a function $f$ on its $0$-mode by
    \begin{equation}
        \mathbb{P}_{=0}(f)=\int_{\T} f(x)\, dx, \notag
    \end{equation}
    and the projection on the $j^{\rm th}$ shell by
    \begin{equation}
        \mathbb{P}_{2^j}(f)(x)=\sum_{\xi\in \Z}\hat{f}(\xi)\varphi_{j}(\xi)e^{2\pi i\xi x} \, . \notag
    \end{equation}
    We also define the frequency cutoff by
    $$
    \mathbb{P}_{\leq 2^j}(f) = \mathbb{P}_{=0}(f) + \sum_{k=1}^j \mathbb{P}_{2^k}(f),
    $$
    the projection onto high frequencies by $\mathbb{P}_{>2^j}(f) = (\operatorname{Id} - \mathbb{P}_{\leq 2^j})(f)$ and the projection off the mean by $\mathbb{P}_{\neq 0}f:=(\operatorname{Id}-\mathbb{P}_{=0})f$. 
\end{definition}

The following lemma is standard, and can be deduced as a consequence of the Poisson summation formula.

\begin{lemma}[\textbf{$L^p$ boundedness of projection operators}]\label{lem:proj}
    $\mathbb{P}_{\leq \lambda}$ is a bounded operator from $L^p$ to $L^p$ for $1 \leq p \leq \infty$ with operator norm independent of $\lambda$.
\end{lemma}

Next we introduce the function spaces which will play a prominent role throughout.

\begin{definition}[\textbf{$\dot H^s$ Sobolev spaces}] For $s \in \mathbb{R}$, we define
    $$\dot{H}^{s}(\mathbb{T})=\left\{ f: \sum_{\xi \in \Z \setminus \{0\}}|\xi|^{2s}|\hat{f}(\xi)|^2<\infty\right\}  $$
    with the norm induced by the sum above.
\end{definition}
\begin{remark}
    For every $f\in \dot H^{s}$ for some $s\in \R$, we define the Fourier coefficients
    \[\hat{f}(\xi)=\int_{\T}f(x)e^{-2\pi i\xi x}\,dx, \hspace{0.25cm} \text{ where } \hspace{0.25cm} \T=[0,1],\]
    and so we can define $\mathbb{P}_{2^j}f$ for $j\geq 0$.  Note that each $\mathbb{P}_{2^j}f$ is smooth for $f\in \dot H^{s}$ irrespective of the value of $s\in \R$.
\end{remark}

\begin{definition}{\textbf{(Homogeneous Besov spaces)}}
    For $\alpha \in \R$, $0<p,q\leq \infty$, and $f \in \mathcal{D}'(\T^d)$, define the homogeneous Besov space to be
    \[
        \dot{B}^{\alpha}_{p,q} = \left\{ f: \|f\|_{\dot{B}^{\alpha}_{p,q}} = \left\| 2^{j\alpha}\|\mathbb{P}_{2^j}f\|_{L^p(\T)}\right\|_{\ell^q_{j\geq 0}} <\infty \right\}.
    \]
\end{definition}
\begin{remark}
    Note for $\alpha > 0$ the space $\dot{B}^{\alpha}_{\infty,\infty}$ is equivalent to $\{f \in C^\alpha : \hat{f}(0) = 0\}$. In addition, the space $\dot{B}^{\alpha}_{2,2}$ is equivalent to the homogeneous Sobolev space $\dot{H}^\alpha$ for all $\alpha \in \R$.
\end{remark}

Following \cite{Gismondi} and \cite{GR}, we utilize a generic intermittent building block which does not solve any stationary PDE. In \cite{GR} this was essential in order to utilize a fully intermittent building block. We do the same here since our setting is the one dimensional torus and we will be using one dimension of intermittency. As was the case in \cite{Gismondi} and \cite{GR}, estimating the error with a homogeneous Sobolev norm ensures that we never require the derivative operator acting on the nonlinearity to "land"; this is in stark contrast to standard convex integration schemes where the error is estimated in either $L^p$ or $C^\alpha$, see \cite{BV2020}. For our purposes, the intermittent slabs from \cite[Lemma 2.12]{Gismondi} will suffice; however it should be noted in this setting the slabs are fully intermittent whereas in \cite{Gismondi} they are not.

\begin{lemma}[\textbf{Intermittent slabs}]\label{lem:boldW}
    Let $\lambda$ be a large power of $2$ and take $0 < \epsilon < 1$ such that $\lambda^{\epsilon}$ is also an integer. Then there exist smooth $\rho_{\lambda,\epsilon}:\T \to \R$ such that
    \begin{enumerate}
        \item\label{w:2} $\int_{\T} \rho_{\lambda,\epsilon} = 0$;
        \item\label{w:3} $\Vert \rho_{\lambda,\epsilon} \Vert_{L^p(\T)} \lesssim \lambda^{(1-\epsilon)\left(\frac{1}{2}-\frac{1}{p}\right)}$;
        \item\label{w:4} $\rho_{\lambda,\epsilon}$ is $\frac{\T}{\lambda^{\epsilon}}$-periodic.
    \end{enumerate}
\end{lemma}

\begin{proof}
    The construction is given in \cite[Lemma 2.12]{Gismondi} just with $d = 1$. For the purposes of the proceeding proof though, we just state the function is of the form
    \begin{equation}\label{eq:rho}
        \rho_{\lambda,\epsilon}(x) = \sum_{n \in \Z} \lambda^{\frac{1-\epsilon_{q+1}}{2}} \phi\left(\lambda x +\lambda^{1-\epsilon}n\right).
    \end{equation}
    where $\phi:\R \to \R$ is smooth, odd, supported in $(-1,1)$, mean-zero, and $L^2$ normalized.
\end{proof}

Again following \cite[Lemma 2.13]{Gismondi} and \cite[Lemma 2.15]{GR} Fourier series expansions of our intermittent building block will prove very useful, with the main application being the proof of Lemma \ref{lem:L^inf_decay_2}.

\begin{lemma}[\textbf{Fourier series representation of the intermittent slab}]\label{lem:fourier}
    The Fourier series representation of $\rho_{\lambda,\epsilon}$ is
    \begin{equation}\label{eq:fourier_series}
        \rho_{\lambda,\epsilon}(x) = \sum_{n \in \Z} \lambda^{\frac{\epsilon-1}{2}} \hat{\phi}\left(\lambda^{\epsilon-1}n\right) e^{2\pi i \lambda^{\epsilon}  n x}.
    \end{equation}
\end{lemma}
\begin{proof}
    Applying the Poisson summation formula to ~\eqref{eq:rho} gives ~\eqref{eq:fourier_series}.
\end{proof}

\begin{lemma}[\textbf{$L^2$ norm of intermittent slab}]\label{lem:L2_norm}
$\rho_{\lambda,\epsilon}$ is $L^2(\T)$ normalized to leading order. More precisely,
$$
\Vert \rho_{\lambda,\epsilon} \Vert_{L^2} = 1 + O\left(\lambda^{\epsilon-1}\right).
$$
\end{lemma}
\begin{proof}
    From Lemma \ref{lem:fourier} and the Plancherel theorem we have the expression
    \begin{equation}\label{eq:Plancherel}
        \Vert \rho_{\lambda,\epsilon} \Vert_{L^2}^2 = \sum_{n \in \Z} \lambda^{\epsilon-1} \left|\hat{\phi}\left(\lambda^{\epsilon-1}n\right)\right|^2.
    \end{equation}
We recognize the right hand side of ~\eqref{eq:Plancherel} as being a Riemann sum approximation of $\Vert \hat{\phi} \Vert_{L^2(\R)}^2$ with step size $\lambda^{\epsilon-1}$. Applying Plancherel one more time, $\Vert \hat{\phi} \Vert_{L^2(\R)} = \Vert \phi \Vert_{L^2(\R)} = 1$ and since $\phi$ is smooth and compactly supported, we have that
\begin{equation}\label{eq:square_ofL2_norm}
    \Vert \rho_{\lambda,\epsilon} \Vert_{L^2(\T)}^2 = 1 + O\left(\lambda^{\epsilon-1}\right).
\end{equation}
Taking the square root of ~\eqref{eq:square_ofL2_norm} and then using the Taylor expansion gives the result.
\end{proof}

The following lemma will show that in every norm that we will be concerned with, the projection of the product of a smooth function and our intermittent slabs can be made arbitrarily small simply by choosing the minimum active frequency of the intermittent slab large enough.

\begin{lemma}[\textbf{High-low products in negative Sobolev norms}]\label{lem:L^inf_decay_2}

    For any $a\in C^\infty(\T)$,
    $$
    \lim_{\lambda \to \infty} \left\Vert \mathbb{P}_{> \lambda^2}(a \rho_{\lambda,\epsilon} )\right\Vert_{L^\infty} = 0.
    $$
\end{lemma}

    \begin{proof}
For $a \in C^\infty(\mathbb{T})$, we begin by expressing its Fourier series as
\begin{equation*}
    a(x)=\sum_{k\in\mathbb{Z}} \widehat{a}(k) e^{2\pi i kx}.
\end{equation*}
Together with \eqref{eq:fourier_series}, this allows us to write
\begin{equation}
    (a\rho_{\lambda,\epsilon})(x)=\sum_{n,k \in \Z} \widehat{a}(k)\lambda^{\frac{\epsilon-1}{2}} \hat{\phi}\left(\lambda^{\epsilon-1}n\right) e^{2\pi i \left(\lambda^{\epsilon}  n+k\right) x}.
\end{equation}
Recall $\lambda$ was chosen to be a large power of $2$. So let us put $\mu = 2\log_2(\lambda) \in \N$. Consequently, from Definition \ref{def:projs} the high-frequency projection has the representation
\begin{equation*}
    \mathbb{P}_{> \lambda^2}(a \rho_{\lambda,\epsilon} )(x)=\sum_{j = \mu}^\infty \sum_{n,k \in \mathbb{Z}} \varphi_j\left(\lambda^{\epsilon}n + k\right) \widehat{a}(k)\lambda^{\frac{\epsilon-1}{2}} \hat{\phi}\left(\lambda^{\epsilon-1}n\right) e^{2\pi i \left(\lambda^{\epsilon}  n+k\right) x},
\end{equation*}
and this gives
\begin{equation*}
    \left\Vert \mathbb{P}_{> \lambda^2}(a \rho_{\lambda,\epsilon} )\right\Vert_{L^\infty} \leq \lambda^{\frac{\epsilon-1}{2}}\sum_{\substack{n,k \in \Z\\ \left|\lambda^{\epsilon}  n+k\right|\geq \frac{1}{2}\lambda^2}} |\widehat{a}(k)| \left|\hat{\phi}\left(\lambda^{\epsilon-1}n\right)\right| \leq  \mathcal{S}_1+\mathcal{S}_2,
\end{equation*}
where we put
\begin{equation*}
    \mathcal{S}_1 \coloneq  \lambda^{\frac{\epsilon-1}{2}} \sum_{|n| \geq \frac{1}{4} \lambda^{2-\epsilon}}\;\sum_{\substack{k \in \Z\\ \left|\lambda^{\epsilon}  n+k\right|\geq \frac{1}{2}\lambda^2}} |\widehat{a}(k)| \left|\hat{\phi}\left(\lambda^{\epsilon-1}n\right)\right|
\end{equation*}
and 
\begin{equation*}
    \mathcal{S}_2\coloneq \lambda^{\frac{\epsilon-1}{2}}\sum_{|n| \leq \frac{1}{4} \lambda^{2-\epsilon}}\;\sum_{\substack{k \in \Z\\ \left|\lambda^{\epsilon}  n+k\right|\geq \frac{1}{2}\lambda^2}} |\widehat{a}(k)| \left|\hat{\phi}\left(\lambda^{\epsilon-1}n\right)\right|.
\end{equation*}
We will estimate each of these terms separately.

\underline{Estimate for $\mathcal{S}_1$}:
Since $a \in C^\infty(\T)$, then $\hat{a} \in \ell^1(\Z)$. Hence
\begin{equation}\label{eq:S1_est}
    \mathcal{S}_1 \lesssim \lambda^{\frac{\epsilon-1}{2}}\sum_{|n| \gtrsim \lambda^{2-\epsilon}}\;\sum_{k \in \Z} |\widehat{a}(k)| \left|\hat{\phi}\left(\lambda^{\epsilon-1}n\right)\right| \lesssim \lambda^{\frac{\epsilon-1}{2}}\sum_{|n| \gtrsim \lambda^{2-\epsilon}}\;\left|\hat{\phi}\left(\lambda^{\epsilon-1}n\right)\right|.
\end{equation}
Now since $\phi$ was assumed to be Schwartz, we have the estimate
\begin{equation*}
    \left|\hat{\phi}(\xi)\right| \lesssim \frac{1}{(1 + |\xi|)^N}
\end{equation*}
for $N > 100$. Hence we have
\begin{equation}\label{eq:sum_hat_phi_est}
    \sum_{|n| \gtrsim \lambda^{2-\epsilon}} \left|\hat{\phi}\left(\lambda^{\epsilon-1}n\right)\right| \lesssim \sum_{|n| \gtrsim \lambda^{2-\epsilon}} \frac{1}{(1 + \left|\lambda^{\epsilon-1}n\right|)^N} \lesssim \lambda^{N(1-\epsilon)}\sum_{|n| \gtrsim \lambda^{2-\epsilon}} |n|^{-N}.
\end{equation}
Now using the integral test we have that
\begin{equation}\label{eq:int_test}
    \sum_{|n| \gtrsim \lambda^{2-\epsilon}} |n|^{-N} \simeq \int_{\lambda^{2-\epsilon}}^\infty x^{-N}\, dx \simeq \lambda^{(1-N)(2-\epsilon)}.
\end{equation}
Combining ~\eqref{eq:S1_est}, ~\eqref{eq:sum_hat_phi_est}, and ~\eqref{eq:int_test} we have the estimate
$$
\mathcal{S}_1 \lesssim \lambda^{\frac{\epsilon-1}{2}} \lambda^{N(1-\epsilon)} \lambda^{(1-N)(2-\epsilon)}
$$
which clearly tends to $0$ as we send $\lambda \to \infty$.

\underline{Estimate of $\mathcal{S}_2$:}
\;

In the regime where $|n| \leq \frac{1}{4}\lambda^{2-\epsilon}$, it follows that 
$\lambda^{\epsilon}|n| \leq \frac{1}{4} \lambda^2$. 
Consequently, from the projection condition we deduce
\[
|k| \geq \left| k+\lambda^{\epsilon} n\right| -\frac{1}{4}\lambda^{\epsilon}|n|\geq \frac{1}{4}\lambda^2 .
\]

Thus, in this region the frequency $k$ is necessarily of high magnitude.

We now exploit the rapid decay of $\widehat{a}(k)$. Since $a \in C^\infty(\mathbb{T})$, its Fourier coefficients decay faster than any polynomial rate. In particular, for $|k|\geq \frac{1}{4} \lambda^2_{q+1}$ (that is, in the high–frequency regime), we have the bound
\[
|\widehat{a}(k)|\lesssim |k|^{-N}
\]
for $N$ again chosen larger than $100$. Therefore, using the trivial estimate
$$
\sum_{|n| \lesssim \lambda^{2-\epsilon}} \left|\hat{\phi}\left(\lambda^{\epsilon-1} n\right)\right| \lesssim \lambda^{2-\epsilon}
$$
we have
\begin{equation}\label{eq:S2_est}
    \mathcal{S}_2 \lesssim \lambda^{\frac{\epsilon-1}{2}}\lambda^{2-\epsilon} \sum_{|k|\geq\frac{1}{4} \lambda^2} |k|^{-N}
\end{equation}
and we once again apply \eqref{eq:int_test}, this time with $\lambda^2$, to finish the estimate in ~\eqref{eq:S2_est} and complete the proof.
\end{proof}

The following lemma will find use in Section \ref{section5} where it can simplify many of the estimates performed.

\begin{lemma}[\textbf{Kato-Ponce-type product estimate}]\label{lem:pseudo_diff_cont}
    Fix $\alpha,\beta \in C^{\infty}(\T^2)$ with $\beta$ having zero mean. Then for $s \geq 0$ we have
    $$
        \Vert \alpha \beta \Vert_{\dot{H}^{-s}} \lesssim \Vert \alpha \Vert_{C^s} \Vert \beta \Vert_{\dot{H}^{-s}}.
    $$
\end{lemma}
\begin{proof}
    See \cite[Lemma 2.16]{GR}\footnote{This proof crucially utilizes \cite[~Proposition 1]{BOZ}}.
\end{proof}

We also collect here some Bernstein type inequalities that will prove useful.

\begin{lemma}[\textbf{Bernstein type estimates}]\label{lem:Bernstein}
    Suppose $u,v:\T \to \R$ are smooth with the frequency support of $u$ contained in $B(0,\mu)$ and the frequency support of $v$ contained in $B(0,2\mu) \setminus B(0,\mu)$, $\mu \gg 1$. Then we have the following estimates:
    \begin{enumerate}
        \item [(a)] For $s \geq 0$ we have
        $$
        \Vert u \Vert_{C^s} \lesssim \mu^s \Vert u \Vert_{L^\infty} \quad \text{and} \quad \Vert u \Vert_{\dot{B}^s_{\infty,\infty}} \lesssim \mu^s \Vert u \Vert_{L^\infty}. 
        $$
        \item [(b)] For $s \in \R$ we have
        $$
        \Vert v \Vert_{\dot{H}^s} \lesssim \mu^s \Vert v \Vert_{L^2} \quad \text{and} \quad \Vert v \Vert_{\dot{B}^s_{\infty,\infty}} \lesssim \mu^s \Vert v \Vert_{L^\infty}.
        $$
    \end{enumerate}
\end{lemma}
\begin{proof}
\noindent\textsf{\underline{Part (a).}}
For the Littlewood--Paley pieces we observe the following:

\begin{itemize}
\item If $2^j \gg \mu$, then $\P_{2^j} u = 0$ since the Fourier support of $u$ is contained in $B(0,\mu)$.

\item If $2^j \lesssim \mu$, then from Lemma \ref{lem:proj} we have
\[
\|\P_{2^j} u\|_{L^\infty} \le \|u\|_{L^\infty}.
\]
\end{itemize}
This give us 
\begin{equation*}
    \|u\|_{\dot{B}^\alpha_{\infty,\infty}} \lesssim \mu^{s} \| \P_{2^j} u\|_{L^\infty}\lesssim\mu^{s}\|u\|_{L^\infty}
\end{equation*}
Let $\chi \in C_c^\infty$ be a smooth radial cutoff at scale $1$ and define $\chi_\mu(\xi) = \chi(\xi/\mu)$.  
Since $u$ has frequency support in $|\xi|\le \mu$, we can write
\begin{equation*}
    u = \left(\sum_{k \in \Z} \chi_\mu(k) e^{2\pi i kx}\right) \ast u.
\end{equation*}

For $|\alpha|\le s$, using Young's convolution inequality we have
\begin{equation*}
    \|\partial^\alpha u\|_{L^\infty} = \left\Vert \partial^\alpha\left(\sum_{k \in \Z} \chi_\mu(k) e^{2\pi i kx}\right)  \ast u\right\Vert_{L^\infty} \le \left\Vert \sum_{k \in \Z} (2\pi i k)^\alpha \chi_\mu(k) e^{2\pi i kx} \right\Vert_{L^1} \|u\|_{L^\infty}.
\end{equation*}
So it suffices to show that
\begin{equation}\label{eq:goal_est}
    \left\Vert \sum_{k \in \Z} (2\pi i k)^\alpha \chi_\mu(k) e^{2\pi i kx} \right\Vert_{L^1} \lesssim \mu^\alpha.
\end{equation}
First, computing the Fourier transform of $x^\alpha \chi_\mu(x)$ we have
$$
\left( x^\alpha \chi_\mu(x)\right)^{\wedge}(\xi) = \mu^{\alpha+1}\int_{\R} x^\alpha \chi(x) e^{2\pi i \mu x \xi}\, dx
$$
Now applying the Poisson summation formula yields
$$
\sum_{k \in \Z} (2\pi i k)^\alpha \chi_\mu(k) e^{-2\pi i kx} = \sum_{k \in \Z} (2\pi i)^\alpha \mu^{\alpha+1} \int_{\R} x^{\alpha} \chi(x) e^{2\pi i \mu (\xi+k)}\, dx.
$$
Applying the $L^1(\T)$ norm we have
$$
\left\Vert \sum_{k \in \Z} (2\pi i k)^\alpha \chi_\mu(k) e^{2\pi i kx} \right\Vert_{L^1} \lesssim \sum_{k \in \Z} \mu^{\alpha+1} \int_0^1 \left|\int_{\R}  x^{\alpha} \chi(x) e^{2\pi i \mu x(\xi+k)}\, dx\right|\, d\xi.
$$
Finally we perform the change of variables $\mu(\xi+k) \mapsto \xi$ to get
\begin{equation*}
    \begin{split}
        \left\Vert \sum_{k \in \Z} (2\pi i k)^\alpha \chi_\mu(k) e^{2\pi i kx} \right\Vert_{L^1} &\lesssim \sum_{k \in \Z} \mu^{\alpha} \int_{\mu k}^{\mu(k+1)} \left|\int_{\R}  x^{\alpha} \chi(x) e^{2\pi i x\xi}\, dx\right|\, d\xi\\
        &= \mu^\alpha \int_{\R} \left|\int_{\R}  x^{\alpha} \chi(x) e^{2\pi i \mu x\xi}\, dx\right|\, d\xi\\
        &\simeq \mu^\alpha \Vert x^\alpha \chi(x) \Vert_{L^1(\R)}\\
        &\lesssim \mu^\alpha.
    \end{split}
\end{equation*}
Now summing over $\alpha \leq s$ and using that $\mu \gg 1$ gives ~\eqref{eq:goal_est}. The modifications for the case $s \not \in \N$ are obvious.

\noindent\textsf{\underline{Part (b).}}
In this setup we observe $\P_{2^j}v \neq 0$ only when $\mu \simeq 2^j$. Then we can write 
\begin{equation*}
     \|v\|_{\dot{B}^\alpha_{\infty,\infty}}\simeq \mu^s\| \P_{2^j} v\|_{L^\infty} \lesssim \mu^s\|v\|_{L^\infty}.
\end{equation*}
For the first estimate in part (b), we have on the support of $\widehat{v}$ that $|k|\sim \mu$ and thus we conclude
\begin{equation*}
\|v\|_{\dot H^s}^2 = \sum_{\mu \le |k| \le 2\mu} |k|^{2s} |\hat v(k)|^2 
\lesssim \mu^{2s} \sum_{\mu \le |k| \le 2\mu} |\hat v(k)|^2 
\le \mu^{2s} \|v\|_{L^2}^2,
\end{equation*}

hence
\begin{equation*}
 \|v\|_{\dot H^s} \lesssim \mu^s \|v\|_{L^2}.
\end{equation*}

\end{proof}

\section{Convex Integration Scheme}\label{section3}
We present here a short outline of the argument to follow. We will consider the relaxed version of Equation ~\eqref{eq:stat_eqn} given by
\begin{equation}\label{eq:relaxed}
\begin{split}
    3\left(u^2\right)' - u^{(3)} = E'
    \end{split}
\end{equation}
for $E : \T \to \R$ which we refer to as the \textit{error}. We assume inductively that we have a smooth classical solution $(u_q,E_q)$ to ~\eqref{eq:relaxed} with $E_q$ not identically $0$.  If we put
\begin{equation}\label{eq:uq+1}
    u_{q+1} = u_q + w_{q+1}
\end{equation}
for some carefully chosen smooth function $w_{q+1}:\T \to \R$, then $E_{q+1}$ must satisfy
\begin{equation}\label{eq:relaxed_eqn}
    \begin{split}
        E_{q+1}' &= 3\left(u_{q+1}^2\right)' - u_{q+1}^{(3)}\\
        &= \left(E_q + 3w_{q+1}^2\right)' + \left(6w_{q+1} u_q\right)'  - w_{q+1}^{(3)}
    \end{split}
\end{equation}

Setting
\begin{equation}\label{eq:osc_error}
    E_O = E_q + 3w_{q+1}^2,
\end{equation}
\begin{equation}\label{eq:Nash_error}
    E_N = 6w_{q+1}u_q
\end{equation}
and
\begin{equation}\label{eq:dis_error}
    E_D = - w_{q+1}''
\end{equation}
we have that
\begin{equation}\label{eq:Rq+1}
    E_{q+1} = E_O + E_N + E_D + C_{q+1}
\end{equation}
for
\begin{equation}\label{eq:Cq+1}
    C_{q+1} =
    \begin{cases}
        1, & E_O + E_N + E_D = 0\\
        0, & \text{otherwise}
    \end{cases}.
\end{equation}
Equations ~\eqref{eq:osc_error} - ~\eqref{eq:dis_error} are referred to as the \textit{oscillation error}, \textit{Nash error}, and \textit{dispersion error}, respectively. We will show that for arbitrary $\alpha < 0$ and $1 \leq p < 2$, both $\Vert w_{q+1} \Vert_{\dot{B}^\alpha_{\infty,\infty}}$ and $\Vert w_{q+1} \Vert_{L^p}$ are summable, and so $u_q \to u \in \dot{B}^\alpha_{\infty,\infty} \cap L^p$. Since $\alpha$ and $p$ are chosen arbitrarily, we get the desired regularity of $u$. Then upon showing that the dispersion, Nash, and oscillation errors all tend to $0$ in $\dot{H}^{-s}$ norm as $q \to \infty$, this will show that $u$ is our desired solution to ~\eqref{eq:stat_eqn} in the sense of Definition \ref{def:weak_para_soln}. We note that the existence of $\mathbb{P}_{\not=0}(u^2)$ as a paraproduct follows nearly immediately from our computations showing that the oscillation error can be made arbitrarily small.

\section{Inductive Proposition}\label{Section4}
We state here the main Proposition \ref{prop:ind}, which we call the inductive proposition, and use it to prove Theorem \ref{thm:main_flex}. In Section \ref{section5} we prove the inductive proposition.

\begin{proposition}[\textbf{Inductive Proposition}]\label{prop:ind}
    There are sequences $u_q$ and $E_q$, $q \geq 0$, such that the following hold:
    \begin{enumerate}
        \item\label{i:1} $(u_q,E_q)$ is a smooth solution to the relaxed equation given by ~\eqref{eq:relaxed}. In addition, $u_q$ has zero mean and $E_q$ is not identically $0$;
        \item\label{i:2} $\Vert E_q \Vert_{\dot{H}^{-s}} < 2^{-q}$;
        \item\label{i:3} There is a constant $C(\alpha,p) > 0$ depending only on $\alpha < 0$ and $1 \leq p < 2$ such that for all $q' \leq q$ we have
        $$
        \Vert u_{q'} - u_{q'-1} \Vert_{\dot{B}^\alpha_{\infty,\infty}} + \Vert u_{q'} - u_{q'-1} \Vert_{L^p} \lesssim 2^{-C(\alpha,p)q'}
        $$
        where the implicit constant depends on $\alpha$ and $p$ but not $q$, or $q'$;
        \item\label{i:4} For all $q' \leq q$ we have that
        $$
        \Vert u_{q'} - u_{q'-1} \Vert_{L^2} \gtrsim 1
        $$
        where the implicit constant is independent of $q$ and $q'$;
        \item\label{i:5} For each $q' \leq q$ there exists a unique $j \in \Z$ such that
        $$
        \mathbb{P}_{2^j}(u_{q'} - u_{q'-1}) = u_{q'} - u_{q'-1};
        $$
        \item\label{i:6} There are $C_1,C_2 > 0$ independent of $q$ such that
        $$
        \sum_{\substack{n,m \leq q\\ n \not  = m}} \Vert (u_n - u_{n-1}) (u_m - u_{m-1}) \Vert_{\dot{H}^{-s}} < C_1 - 2^{-q}
        $$
        and
        $$
         \sum_{n \leq q} \Vert (u_n - u_{n-1})^2 \Vert_{\dot{H}^{-s}} < C_2 - 2^{-q + 100}.
        $$
    \end{enumerate}
\end{proposition}

\begin{proof}[Proof of Theorem~\ref{thm:main_flex} using Proposition~\ref{prop:ind}] We first verify the base case of the induction proposition. Put $u_0 = A\sin(x)$ and $R_0 = 3A^2\sin^2(x) + A\sin(x)$ for some $A > 0$ to be chosen later. Then clearly \ref{i:1}, \ref{i:4}, and \ref{i:5} (with $u_{-1} = 0$\footnote{If one desired a solution with mean $m$ simply setting $u_{-1} = m$ and then proceeding with the rest of the scheme will produce the desired solution. Since the mean of solutions of KdV are conserved in time, in some sense it suffices to only consider mean zero solutions.}) hold. \ref{i:2} and \ref{i:3} hold upon choosing $A$ small enough. Finally \ref{i:6} holds by choosing $C_1$ and $C_2$ large enough.

Now we assume Proposition \ref{prop:ind} holds. \ref{i:3} implies that $\{u_q\}$ is Cauchy in both $L^p$ and $\dot{B}^\alpha_{\infty,\infty}$ for all $p < 2$ and $\alpha < 0$, and thus
$$
u_q \to u \in \bigcap_{\epsilon > 0} \dot{B}^{-\epsilon}_{\infty,\infty} \cap L^{2-\epsilon}.
$$
From \ref{i:5}, the frequency support of each $u_q - u_{q-1}$ is disjoint, and thus from \ref{i:4} we have
$$
\Vert u \Vert_{L^2}^2 \geq \Vert u_0 \Vert_{L^2}^2 + \sum_{q' \leq q} \Vert u_{q'} - u_{q'-1} \Vert_{L^2}^2 \gtrsim \Vert u_0 \Vert_{L^2}^2 + q.
$$
Sending $q \to \infty$ shows that $u \not \in L^2$.

Now we show that $u$ solves ~\eqref{eq:stat_eqn} in the sense of Definition \ref{def:weak_para_soln}. By \ref{i:1}, we have that
\begin{equation}\label{eq:weak_form_relaxed}
    -\int_{\T} u_q \psi^{(3)} + 3\int_{\T} u_q^2 \psi' = -\int_\T E_q \psi'
\end{equation}
for all $\psi \in C^\infty(\T)$. From \ref{i:2}, sending $q \to \infty$ we get
\begin{equation}\label{eq:Eq_to_0}
    \left|\int_{\T} E_q \psi' \right| \lesssim \Vert E_q \Vert_{\dot{H}^{-s}} \Vert \psi' \Vert_{\dot{H}^s} < 2^{-q} \Vert \psi' \Vert_{\dot{H}^s} \to 0.
\end{equation}
Similarly since $u_q \to u$ in $L^1$ we have
\begin{equation}\label{eq:nonlin_term}
    \int_\T u_q \psi^{(3)} \to \int_\T u \psi^{(3)} = \langle u,\psi^{(3)}\rangle_{\dot{H}^{-s'},\dot{H}^{s'}}
\end{equation}
for all $s' > 0$. Finally combining \ref{i:5} and \ref{i:6} shows that $\mathbb{P}_{\not=0}(u^2)$ exists in the sense of Definition \ref{def:paras}. Thus
\begin{equation}\label{eq:non_lin_term}
    \begin{split}
        \int_\T u_q^2 \psi' &= \int_\T \mathbb{P}_{\not=0}(u_q^2) \psi'\\
        &= \int_\T \left(\sum_{q',q'' \leq q} \mathbb{P}_{\not=0}\left((u_{q'} - u_{q'-1})(u_{q''} - u_{q''-1})\right)\right)\psi'\\
        &= \left\langle \sum_{q',q'' \leq q} \mathbb{P}_{\not=0}\left((u_{q'} - u_{q'-1})(u_{q''} - u_{q''-1})\right),\psi'\right\rangle_{\dot{H}^{-s},\dot{H}^s}\\
        &\to \langle \mathbb{P}_{\not=0}(u^2), \psi'\rangle_{\dot{H}^{-s},\dot{H}^s}
    \end{split}
\end{equation}
Combining ~\eqref{eq:weak_form_relaxed}, ~\eqref{eq:Eq_to_0}, ~\eqref{eq:nonlin_term}, and ~\eqref{eq:non_lin_term} shows $u$ is our desired solution.

\end{proof}

\section{Proof of Proposition \ref{prop:ind}}\label{section5}
\subsection{Parameters}\label{sec:param}
Here we specify the parameters that we will utilize throughout this section. First define the sequence $\{\epsilon_q\}_{q \in \N}$ by
$$
\epsilon_q = 1 - 2^{-q-1000}.
$$
Next we assume that $\{\lambda_q\}_{q \in \N}$ is a monotonically increasing sequence of very large powers of $2$ such that $\lambda_q^{\epsilon_q}$ is also an integer for all $q$. The size of $\lambda_{q+1}$ will be chosen to ensure that \eqref{eq:sigma}, \eqref{eq:dis_error_est}, \eqref{eq:Nash_error_est}, \eqref{eq:second_term_osc_error}, \eqref{eq:osc_first_term_est}, \eqref{eq:osc_second_term_est}, \eqref{eq:osc_third_term_est}, \eqref{eq_itm4}, \eqref{eq_itm4_2}, \eqref{eq:freq_cutoff_ineq}, \eqref{eq:RL_lemma}, \eqref{eq:shell_cond}, \eqref{eq:off_diag_est1} are all satisfied. Finally define $\{\sigma_q\}_{q \in \N}$ by
\begin{equation}\label{eq:sigma}
    \sigma_q = \frac{5}{4}\lambda_q^3 \in \Z.
\end{equation}
 Note we can ensure $\sigma_q \in \Z$ by requiring that $\lambda_q$ is a large enough power of $2$.

\subsection{Construction of Increment}

\begin{definition}[\textbf{Increment}]\label{def:wq+1}
    Put $\rho_{q+1} = \rho_{\lambda_{q+1},\epsilon_{q+1}}$ from Lemma \ref{lem:boldW} and let
    \begin{equation}\label{eq:wq+1}
        w_{q+1} = \sqrt{\frac{2}{3}}\mathbb{P}_{\leq \lambda_{q+1}^2} \left(a_{q}(x) \rho_{q+1}(x)\right) \cos(2\pi \sigma_{q+1} x)
    \end{equation}
    where we put
    \begin{equation}\label{eq:ak}
        a_{q} = \left(2\Vert E_q \Vert_{L^\infty} - E_q\right)^{1/2}.
    \end{equation}
\end{definition}



\begin{lemma}[\textbf{$L^p$ estimates for $w_{q+1}$}]\label{lem:wq+1_est}
    For all $1 \leq p \leq \infty$ we have
    $$
    \Vert w_{q+1} \Vert_{L^p} \lesssim \lambda_{q+1}^{(1-\epsilon_{q+1})\left(\frac{1}{2}-\frac{1}{p}\right)}.
    $$
\end{lemma}
\begin{proof}
    Upon applying Lemma \ref{lem:proj}, H\"{o}lder inequality, Lemma \ref{lem:boldW}, and ~\eqref{eq:wq+1} we have
    $$
    \Vert w_{q+1} \Vert_{L^p} \lesssim \Vert \rho_{q+1} \Vert_{L^p} \lesssim \lambda_{q+1}^{(1-\epsilon_{q+1})\left(\frac{1}{2} - \frac{1}{p}\right)}
    $$
    establishing the lemma.
\end{proof}
\begin{remark}
    Note in the proof of Lemma \ref{lem:wq+1_est} we used that $\Vert a_q \Vert_{L^\infty} \lesssim 1$. This is due to the fact that this norm is independent of the parameter $\lambda_{q+1}$ (see ~\eqref{eq:ak}), and so any large constants which depend only on $\lambda_1,\ldots, \lambda_q$ can be dominated by choosing $\lambda_{q+1}$ large enough so long as there is decay in this parameter $\lambda_{q+1}$. We will perform this procedure quite often throughout the rest of the text; any constants which only depend on $\lambda_q$ will be estimated by $1$.
\end{remark}

\subsection{Proof of Item \ref{i:1}}
Let us suppose $u_q$ is smooth and has mean zero. Then from ~\eqref{eq:uq+1} and Definition \ref{def:wq+1} we see $u_{q+1}$ is also smooth and has zero mean. From our choice of $C_{q+1}$ in ~\eqref{eq:Cq+1}, we deduce $E_{q+1}$ is not identically $0$. From ~\eqref{eq:relaxed_eqn} we see that $E_{q+1}$ is smooth and $(u_{q+1},E_{q+1})$ solve ~\eqref{eq:relaxed}.

\subsection{Proof of Item \ref{i:2}}

\noindent\texttt{Dispersion error: }

Since 
\begin{equation*}
    E_D=-w_{q+1}''
\end{equation*}
we have on Fourier side 
\begin{equation*}
\widehat{E}_D(k) = - (2\pi i k)^2 \, \widehat{w}_{q+1}(k) = (2\pi)^2 k^2 \, \widehat{w}_{q+1}(k).
\end{equation*}
Then using Lemma \ref{lem:wq+1_est} and our assumption $s > 0$ is large enough (larger than $5/2$ suffices) we have
\begin{equation}\label{eq:dis_error_est}
    \|E_D \|^2_{\dot{H}^{-s}(\mathbb{T})} \simeq \sum_{k\neq 0} |k|^{-2s}|k|^4|\widehat{w}_{q+1}(k)|^2\lesssim \lambda_{q+1}^{\epsilon_{q+1}-1} < 2^{-2q-200}
\end{equation}
where the final estimate in ~\eqref{eq:dis_error_est} holds by choosing $\lambda_{q+1}$ large enough.

\noindent\texttt{Nash error: }

Recall \eqref{eq:Nash_error} and the Sobolev embedding $L^1 \subset \dot{H}^{-s}$ for $s > 1/2$ (recall we choose $s$ large). Then using Lemma \ref{lem:wq+1_est} we have
\begin{equation}\label{eq:Nash_error_est}
\| E_N \|_{\dot H^{-s}} \lesssim \| E_N \|_{L^1} 
\lesssim \| w_{q+1} \|_{L^1} \, \| u_q \|_{L^\infty}\lesssim\lambda_{q+1}^{(1-\epsilon_{q+1})\left(-\frac{1}{2}\right)} < 2^{-q-100} .
\end{equation} 
where once again the final inequality in ~\eqref{eq:Nash_error_est} follows from choosing $\lambda_{q+1}$ large enough.

\noindent\texttt{Oscillation error: }

We start with the expression from Definition \ref{def:wq+1}
\begin{equation*}
     \sqrt{\frac{3}{2}}w_{q+1} = \mathbb{P}_{\leq \lambda_{q+1}^2} (a_q \rho_{q+1}) \cos(2 \pi \sigma_{q+1} x).
\end{equation*}

Using the trigonometric identity \(\cos^2 \theta = \frac{1}{2} + \frac{1}{2} \cos(2 \theta)\), we write
\begin{equation*}
    \frac{3}{2} w_{q+1}^2 = \left[\mathbb{P}_{\leq \lambda_{q+1}^2} (a_q \rho_{q+1}) \right]^2 \left[\frac{1}{2} + \frac{1}{2} \cos(4 \pi \sigma_{q+1} x)\right].
\end{equation*}

Multiplying both sides by 2, we get
\begin{equation*}
    3 w_{q+1}^2 = \left[\mathbb{P}_{\leq \lambda_{q+1}^2} (a_q \rho_{q+1}) \right]^2 + \left[\mathbb{P}_{\leq \lambda_{q+1}^2} (a_q \rho_{q+1}) \right]^2 \cos(4 \pi \sigma_{q+1} x)
\end{equation*}
and hence we can write 
\begin{equation}\label{eq:osc_error_decomp}
    E_q+3 w_{q+1}^2 = \left( E_q+\left[\mathbb{P}_{\leq \lambda_{q+1}^2} (a_q \rho_{q+1}) \right]^2 \right) + \left[\mathbb{P}_{\leq \lambda_{q+1}^2} (a_q \rho_{q+1}) \right]^2 \cos(4 \pi \sigma_{q+1} x).
\end{equation}
Let us start with the second term on the right had side of ~\eqref{eq:osc_error_decomp}. Here we invoke Lemma \ref{lem:pseudo_diff_cont} and Lemma \ref{lem:Bernstein} to get
\begin{equation*}
    \begin{split}
        \left\| \left[\mathbb{P}_{\leq \lambda_{q+1}^2} (a_q \rho_{q+1}) \right]^2 \cos(4 \pi \sigma_{q+1} x) \right\|_{{\dot H^{-s}}} &\lesssim \left\|\left[\mathbb{P}_{\leq \lambda_{q+1}^2} (a_q \rho_{q+1}) \right]^2 \right\|_{C^s} \| \cos(4\pi \sigma_{q+1}x)\|_{{\dot H^{-s}}}\\
        &\lesssim \lambda_{q+1}^{2s} \lambda_{q+1}^{2(1-\epsilon_{q+1})\frac{1}{2}} \; \lambda_{q+1}^{-3s}\\
    &\lesssim \lambda_{q+1}^{-s+1-\epsilon_{q+1}}.
    \end{split}
\end{equation*}
Since $s > 0$ is large, for $\lambda_{q+1}$ large enough we have that
\begin{equation}\label{eq:second_term_osc_error}
    \left\| \left[\mathbb{P}_{\leq \lambda_{q+1}^2} (a_q \rho_{q+1}) \right]^2 \cos(4 \pi \sigma_{q+1} x) \right\|_{{\dot H^{-s}}} < 2^{-q-1000}.
\end{equation}

Recall from Lemma \ref{lem:L2_norm} (specifically ~\eqref{eq:square_ofL2_norm}) we have that
$$
\mathbb{P}_{=0}\left(\rho_{q+1}^2\right) = 1 + O\left(\lambda_{q+1}^{\epsilon_{q+1}-1}\right).
$$
So let us define $|\tilde{C}_{q+1}| \lesssim 1$ such that
\begin{equation}\label{eq:tilde_Cq+1}
    \mathbb{P}_{=0}\left(\rho_{q+1}^2\right) = 1 + \tilde{C}_{q+1}\lambda_{q+1}^{\epsilon_{q+1}-1}.
\end{equation}
Now analyzing the first term on the right hand side of ~\eqref{eq:osc_error_decomp}, we write

\begin{equation}\label{eq:Error_cancellation}
    \begin{split}
    E_q+\left[\mathbb{P}_{\leq \lambda_{q+1}^2} (a_q \rho_{q+1}) \right]^2  &= E_q + \left[a_q\rho_{q+1}-\mathbb{P}_{> \lambda_{q+1}^2} (a_q \rho_{q+1})\right]^2\\
    &= E_q + (a_q \rho_{q+1})^2 - 2 \, a_q \rho_{q+1} \, \mathbb{P}_{> \lambda_{q+1}^2} (a_q \rho_{q+1}) + \left[\mathbb{P}_{> \lambda_{q+1}^2} (a_q \rho_{q+1})\right]^2\\
    &=E_q+(2\|E_q\|_{L^\infty}-E_q )[1 + \tilde{C}_{q+1}\lambda_{q+1}^{\epsilon_{q+1}-1} +\mathbb{P}_{\neq 0} (\rho_{q+1}^2) ]\\
    &\qquad- 2 \, a_q \rho_{q+1} \, \mathbb{P}_{> \lambda_{q+1}^2} (a_q \rho_{q+1}) + \left[\mathbb{P}_{> \lambda_{q+1}^2} (a_q \rho_{q+1})\right]^2\\
    &=E_q + (2\|E_q\|_{L^\infty}-E_q ) + \left(\tilde{C}_{q+1} \lambda_{q+1}^{\epsilon_{q+1}-1} + 1\right) a_q^2 \mathbb{P}_{\not=0}(\rho_{q+1}^2)\\
    &\qquad- 2 \, a_q \rho_{q+1} \, \mathbb{P}_{> \lambda_{q+1}^2} (a_q \rho_{q+1}) + \left[\mathbb{P}_{> \lambda_{q+1}^2} (a_q \rho_{q+1})\right]^2\\
    &=  2\|E_q\|_{L^\infty}+\left(\tilde{C}_{q+1}\lambda_{q+1}^{\epsilon_{q+1}-1}+1\right) a_q^2 \P_{\neq0}({\rho_{q+1}^2})\\
    &\qquad- 2 \, a_q \rho_{q+1} \, \mathbb{P}_{> \lambda_{q+1}^2} (a_q \rho_{q+1}) + \left[\mathbb{P}_{> \lambda_{q+1}^2} (a_q \rho_{q+1})\right]^2
    \end{split}
\end{equation}
Let us take the homogeneous Sobolev norm of ~\eqref{eq:Error_cancellation} and analyze in three parts.
\begin{itemize}
    \item 
    The first term we analyze will be
    \begin{equation*}
        2\Vert E_q \Vert_{L^\infty} + \left(\tilde{C}_{q+1} \lambda_{q+1}^{\epsilon_{q+1}-1} + 1\right) a_q^2 \P_{\neq0}({\rho_{q+1}^2}).
    \end{equation*}
    Applying the $\dot{H}^{-s}$ norm we see the $2\Vert E_q \Vert_{L^\infty}$ term vanishes since it is constant. So making use of ~\eqref{eq:tilde_Cq+1}, Lemma \ref{lem:pseudo_diff_cont}, and Lemma \ref{lem:wq+1_est} we have
\begin{equation}\label{eq:osc_first_term}
    \begin{split}
    \left \|2 \Vert E_q \Vert_{L^\infty} +  \left(\tilde{C}_{q+1} \lambda_{q+1}^{\epsilon_{q+1}-1} + 1\right)a_ q^2 \P_{\neq0}({\rho_{q+1}^2}) \right\|_{\dot{H}^{-s}} 
    &\lesssim \| a_q\|_{C^s} \| \P_{\neq0}({\rho_{q+1}^2})\|_{\dot{H}^{-s}(\mathbb{T})} 
    \\
    &\lesssim \| \P_{\neq 0}(\rho_{q+1}^2) \|_{\dot{H}^{-s}}\\
    &\lesssim \lambda_{q+1}^{-s\epsilon_{q+1}}
    \|\rho_{q+1}^2\|_{L^2}\\
    &\lesssim \lambda_{q+1}^{-s\epsilon_{q+1}} \|\rho_{q+1}\|^2_{L^4}\\
    &\lesssim \lambda_{q+1}^{-s\epsilon_{q+1}}\lambda_{q+1}^{\frac{\1-\epsilon_{q+1}}{2}}.
    \end{split}
\end{equation}
We have decay in ~\eqref{eq:osc_first_term} so long as
$$
\frac{1}{2s+1} < \epsilon_{q+1}.
$$
From our choices of $s > 5/2$ and $\epsilon_{q+1} = 1 - 2^{-q-1001}$ this is clearly satisfied. Hence we may take $\lambda_{q+1}$ large enough to ensure that
\begin{equation}\label{eq:osc_first_term_est}
    \left \|2 \Vert E_q \Vert_{L^\infty} +  \left(\tilde{C}_{q+1} \lambda_{q+1}^{\epsilon_{q+1}-1} + 1\right)a_ q^2 \P_{\neq0}({\rho_{q+1}^2}) \right\|_{\dot{H}^{-s}} < 2^{-q-1000}.
\end{equation}

\item The second term we analyze is
$$
\left[\mathbb{P}_{> \lambda_{q+1}^2} (a_q \rho_{q+1})\right]^2,
$$
and for this we invoke Lemma \ref{lem:L^inf_decay_2} to get
\begin{equation}\label{eq:osc_second_term_est}
    \begin{split}
        \left\|[\mathbb{P}_{> \lambda_{q+1}^2} (a_q \rho_{q+1}) ]^2\right\|_{\dot{H}^{-s}(\mathbb{T})} &\lesssim \left\| [\mathbb{P}_{> \lambda_{q+1}^2} (a_q \rho_{q+1}) ]^2\right\|_{L^1}\\
    &\lesssim \left\| \mathbb{P}_{> \lambda_{q+1}^2} (a_q \rho_{q+1}) \right\|^2_{L^\infty}\\
    &< 2^{-q-1000}
    \end{split}
\end{equation}
where the final inequality in ~\eqref{eq:osc_second_term_est} holds for large enough $\lambda_{q+1}$.

\item For the last term again using Lemma \ref{lem:L^inf_decay_2} we have
\begin{equation}\label{eq:osc_third_term_est}
    \begin{split}
        \left\|a_q \rho_{q+1} \, \mathbb{P}_{> \lambda_{q+1}^2} (a_q \rho_{q+1}) \right\|_{\dot{H}^{-s}} &\lesssim \left\|a_q \rho_{q+1} \, \mathbb{P}_{> \lambda_{q+1}^2} (a_q \rho_{q+1}) \right\|_{L^1}\\
    &\lesssim \|a_q\rho_{q+1}\|_{L^2} \; \| \P_{>\lambda_{q+1}^2}(a_q\rho_{q+1}) \|_{L^2}\\
    &\lesssim  \|a_q\|_{L^\infty} \|\rho_{q+1}\|_{L^2} \| \P_{>\lambda_{q+1}^2}(a_q\rho_{q+1}) \|_{L^2}\\
    &\lesssim \| \P_{>\lambda_{q+1}^2}(a_q\rho_{q+1}) \|_{L^2}\\
    &< 2^{-q-1000}.
    \end{split}
\end{equation}
where again the final estimate in ~\eqref{eq:osc_third_term_est} requires $\lambda_{q+1}$ to be chosen large enough.
\end{itemize}

Hence combining ~\eqref{eq:Rq+1}, ~\eqref{eq:osc_error_decomp}, ~\eqref{eq:second_term_osc_error}, ~\eqref{eq:Error_cancellation}, ~\eqref{eq:osc_first_term_est}, ~\eqref{eq:osc_second_term_est}, and ~\eqref{eq:osc_third_term_est} we arrive at
\begin{equation}\label{eq:osc_error_est}
    \Vert R_O \Vert_{\dot{H}^{-s}} < 4\left(2^{-q-1000}\right) < 2^{-q-100}.
\end{equation}

And so finally from ~\eqref{eq:dis_error_est}, ~\eqref{eq:Nash_error_est}, and ~\eqref{eq:osc_error_est} we conclude
$$
\Vert E_{q+1} \Vert_{\dot{H}^{-s}} < 3\left(2^{-q-100}\right) < 2^{-q-1} 
$$
which completes the induction.

\subsection{Proof of Item \ref{i:3}}  
From Lemma \eqref{lem:wq+1_est} we have 
    \begin{equation}\label{eq_itm4}
        \| u_{q+1}-u_{q}\|_{L^p}= \| w_{q+1}\|_{L^p} \lesssim \lambda_{q+1}^{(1-\epsilon_{q+1})\left(\frac{1}{2} - \frac{1}{p}\right)} <2^{-C'(p)q}
    \end{equation}
for $C'(p)=\frac{1}{2}-\frac{1}{p}$ and large enough $\lambda_{q+1}$. Now suppose $q(\alpha)$ is the final index such that 
\begin{equation*}
    \alpha+\frac{1-\epsilon_{q(\alpha)}}{2} \geq \frac{\alpha}{2}.
\end{equation*}

We divide the analysis into two cases:
\begin{itemize}
    \item $q+1\leq q(\alpha)$, then we have
    \begin{equation}\label{eq_itm4_1}
        \|u_{q+1}-u_q\|_{\dot{B}^{\alpha}_{\infty,\infty}}=\|w_{q+1}\|_{\dot{B}^{\alpha}_{\infty,\infty}} \leq \left( 2^{q(\alpha)} \max_{0 \leq n\leq q(\alpha)} \|w_n\|_{\dot{B}^{\alpha}_{\infty,\infty}}\right)2^{-q+1}.
    \end{equation}
    We note the constant in \eqref{eq_itm4_1} depends solely on $\alpha$.
    \item $q+1 > q(\alpha)$ then using Lemma \ref{lem:Bernstein} we have 
    \begin{equation}\label{eq_itm4_2}
        \| u_{q+1}-u_q\|_{\dot{B}^{\alpha}_{\infty,\infty}}=\|w_{q+1}\|_{\dot{B}^{\alpha}_{\infty,\infty}}\lesssim \lambda_{q+1}^\alpha \|w_{q+1}\|_{L^\infty}\lesssim\lambda_{q+1}^{\alpha+\frac{1-\epsilon_{q+1}}{2}}\leq 2^{-C''(\alpha)q}
    \end{equation}
    for $C''(\alpha)=-\frac{\alpha}{2}$ and $\lambda_{q+1}$ large enough.
\end{itemize}

From ~\eqref{eq_itm4}, ~\eqref{eq_itm4_1} and ~\eqref{eq_itm4_2} we obtain
\begin{equation*}
    \| u_{q+1}-u_q\|_{\dot{B}^{\alpha}_{\infty,\infty}}+\| u_{q+1}-u_{q}\|_{L^p} \leq 2^{-C'(p)q}+\left( 2^{q(\alpha)+1} \max_{0 \leq n\leq q(\alpha)} \|w_n\|_{\dot{B}^{\alpha}_{\infty,\infty}}\right)2^{-q}+2^{-C''(\alpha)q}
\end{equation*}
hence the result with a suitable choice of constant $C(\alpha,p)$.

\subsection{Proof of Item \ref{i:4}}
We need to show that
\begin{equation}\label{eq:wq+1_L2_lowr_bd}
    \Vert w_{q+1} \Vert_{L^2} \gtrsim 1.
\end{equation}
From ~\eqref{eq:ak} and \ref{i:1} we see that
$$
0 < \left(\Vert E_q \Vert_{L^\infty}\right)^{1/2} \leq |a_q| \leq \left(3\Vert E_q \Vert_{L^\infty}\right)^{1/2}
$$
and so
\begin{equation}\label{eq:aq_sim_1}
    |a_q| \simeq 1 \quad \text{and} \quad \Vert a_q \Vert_{L^\infty} \simeq 1.
\end{equation}
Now using the $L^2$ normalization of $\rho_{q+1}$ and ~\eqref{eq:aq_sim_1} we have that
\begin{equation}\label{eq:lwr_bd_arho}
    1 = \Vert \rho_{q+1} \Vert_{L^2} \leq \Vert a_q \rho_{q+1} \Vert_{L^2} \Vert a_q \Vert_{L^\infty}^{-1} \simeq \Vert a_q \rho_{q+1} \Vert_{L^2}.
\end{equation}
Now clearly we have
$$
\Vert a_q \rho_{q+1} \Vert_{L^2} \leq \left\Vert \mathbb{P}_{\leq \lambda_{q+1}^2}(a_q \rho_{q+1})\right\Vert_{L^2} + \left\Vert \mathbb{P}_{\geq \lambda_{q+1}^2}(a_q \rho_{q+1})\right\Vert_{L^2}
$$
and from Lemma \ref{lem:L^inf_decay_2} we may choose $\lambda_{q+1}$ large enough to ensure that
\begin{equation}\label{eq:freq_cutoff_ineq}
    \left\Vert \mathbb{P}_{\geq \lambda_{q+1}^2}(a_q \rho_{q+1})\right\Vert_{L^2} < \left\Vert \mathbb{P}_{\leq \lambda_{q+1}^2}(a_q \rho_{q+1})\right\Vert_{L^2}.
\end{equation}
Combining ~\eqref{eq:lwr_bd_arho} and ~\eqref{eq:freq_cutoff_ineq} we see that
\begin{equation}\label{eq:lwr_bd_no_cos}
    1 \lesssim \left\Vert \mathbb{P}_{\leq \lambda_{q+1}^2}(a_q \rho_{q+1})\right\Vert_{L^2}.
\end{equation}
Now we compute
\begin{equation}\label{eq:L2_norm}
    2\left\Vert \mathbb{P}_{\leq \lambda_{q+1}^2}(a_q \rho_{q+1}) \cos(2\pi \sigma_{q+1}x) \right\Vert_{L^2}^2 = \left\Vert \mathbb{P}_{\leq \lambda_{q+1}^2}(a_q \rho_{q+1}) \right\Vert_{L^2}^2 + \int_{\T} \mathbb{P}_{\leq \lambda_{q+1}^2}(a_q \rho_{q+1})^2\, \cos(4\pi \sigma_{q+1}x).
\end{equation}
Now analyzing the final expression in ~\eqref{eq:L2_norm} we see using the Riemann-Lebesgue lemma that for $\lambda_{q+1}$ large enough we have that
\begin{equation}\label{eq:RL_lemma}
    \int_{\T} \mathbb{P}_{\leq \lambda_{q+1}^2}(a_q \rho_{q+1})^2\, \cos(4\pi \sigma_{q+1}x) > -\frac{1}{2}\left\Vert \mathbb{P}_{\leq \lambda_{q+1}^2}(a_q \rho_{q+1}) \right\Vert_{L^2}^2.
\end{equation}
Combining ~\eqref{eq:lwr_bd_no_cos}, ~\eqref{eq:L2_norm}, and ~\eqref{eq:RL_lemma} gives
$$
\left\Vert \mathbb{P}_{\leq \lambda_{q+1}^2}(a_q \rho_{q+1}) \cos(2\pi \sigma_{q+1}x) \right\Vert_{L^2} > \frac{1}{2} \left\Vert \mathbb{P}_{\leq \lambda_{q+1}^2}(a_q \rho_{q+1}) \right\Vert_{L^2} \gtrsim 1
$$
which in turn gives ~\eqref{eq:wq+1_L2_lowr_bd}.

\subsection{Proof of Item \ref{i:5}}
Recall that multiplication by $\cos(2\pi\sigma_{q+1}x)$ corresponds to a translation on the frequency side by $\pm\sigma_{q+1}$. Hence from Definition \ref{def:projs} and Definition \ref{def:wq+1} we see
\begin{equation*}
\operatorname{supp}\widehat{w}_{q+1}
\subset
B(\sigma_{q+1},2\lambda_{q+1}^2)
\cup
B(-\sigma_{q+1},2\lambda_{q+1}^2).
\end{equation*}
To ensure this support is contained in the portion of the smooth Littlewood-Paley frequency projection which is identity:
\begin{equation*}
\{\,2^j \le |\xi| \le \tfrac{12}{7}2^j\,\},
\end{equation*}
(recall again Definition \ref{def:projs}) it suffices (assuming $\lambda_{q+1}^3 = 2^j$ for large  $j\geq 0$) that
\begin{equation*}
2^j
\le
\frac{5}{4}\lambda_{q+1}^3-2\lambda_{q+1}^2 \quad \text{and} \quad
\frac{5}{4}\lambda_{q+1}^3+2\lambda_{q+1}^2
\le
\tfrac{12}{7}2^j.
\end{equation*}
Dividing by $\lambda_{q+1}^2$ and rearranging terms gives
\begin{equation}\label{eq:shell_cond}
2 \le \frac{1}{4}\lambda_{q+1} \quad \text{and} \quad
2 \le \frac{13}{28}\lambda_{q+1}.
\end{equation}
By choosing $\lambda_{q+1}$ sufficiently large ~\eqref{eq:shell_cond} holds and thus
\begin{equation}\label{eq:shell_identity}
    \mathbb{P}_{\lambda_{q+1}^3}(w_{q+1}) = w_{q+1}.
\end{equation}
~\eqref{eq:shell_identity} gives item \ref{i:5}.

\subsection{Proof of Item \ref{i:6}}
We start with induction hypothesis
\begin{equation}\label{eq:ind_hyp}
     \sum_{\substack{n,m \leq q\\ n \not  = m}} \Vert (u_n - u_{n-1}) (u_m - u_{m-1}) \Vert_{\dot{H}^{-s}} < C_1 - 2^{-q}.
\end{equation}
From Lemma \ref{lem:wq+1_est}, we may choose $\lambda_{q+1}$ large enough such that
\begin{equation}\label{eq:off_diag_est1}
    2\|w_{q+1}\|_{L^1} \sum_{n\leq q} \|u_n-u_{n+1}\|_{L^\infty} < 2^{-q-100}
\end{equation}
so then using ~\eqref{eq:ind_hyp} and ~\eqref{eq:off_diag_est1} we may write 
\begin{equation}\label{eq:off_diag_est2}
\begin{split}
    \sum_{\substack{n,m \leq q+1\\ n \not  = m}} \Vert (u_n - u_{n-1}) (u_m - u_{m-1}) \Vert_{\dot{H}^{-s}} &< \sum_{\substack{n,m \leq q\\ n \not  = m}} \Vert (u_n - u_{n-1}) (u_m - u_{m-1}) \Vert_{\dot{H}^{-s}}\\   
    &\qquad    + 2 \sum_{n\leq q } \|(u_n-u_{n-1})w_{q+1}\|_{\dot{H}^{-s}}\\
    &< C_1-2^{-q} +2\|w_{q+1}\|_{L^1} \sum_{n\leq q} \|u_n-u_{n+1}\|_{L^\infty} \\
    &<C_1-2^{-q} +2^{-q-100}\\
    &<  C_1-2^{-q-1}.
\end{split}    
\end{equation}
This gives the first portion of item \ref{i:6}. For the next part we again use inductive hypothesis
\begin{equation}\label{eq:ind_hyp2}
    \sum_{n \leq q} \Vert (u_n - u_{n-1})^2 \Vert_{\dot{H}^{-s}} < C_2 - 2^{-q + 100}
\end{equation}
as well as item \ref{i:2} and ~\eqref{eq:osc_error_est} to deduce
\begin{equation}
    \begin{split}
         \sum_{n \leq q+1} \Vert (u_n - u_{n-1})^2 \Vert_{\dot{H}^{-s}}
    &< C_2 - 2^{-q + 100} + \| w_{q+1}^2\|_{\dot{H}^{-s}}\\
    &= C_2 - 2^{-q + 100}+\frac{1}{3}||E_q + 3w_{q+1}^2 - E_q||_{\dot{H}^{-s}}\\
    &\leq C_2 - 2^{-q + 100}+\frac{1}{3}||E_q + 3w_{q+1}^2||_{\dot{H}^{-s}} + \frac{1}{3}||E_q||_{\dot{H}^{-s}}\\
    &< C_2 - 2^{-q + 100} + 2^{-q-101}+2^{-q-1}\\
    &< C_2 -2^{-q+99}.
    \end{split}
\end{equation}
This completes the proof of item \ref{i:6} and thus Proposition \ref{prop:ind}.

\noindent\textsc{Department of Mathematics, Purdue University, West Lafayette, IN, USA.}
\vspace{.03in}
\newline\noindent\textit{Email address}: \href{mailto:pathak30@purdue.edu}{pathak30@purdue.edu}.

\end{document}